\documentclass[12pt]{article}
\usepackage{amscd, amssymb, amsmath, amsthm, amsfonts}
\usepackage{latexsym, graphics, graphicx,psfrag}
\usepackage[all]{xy}

\newtheorem{thm}{Theorem}[section]
\newtheorem{col}[thm]{Corollary}
\newtheorem{lem}[thm]{Lemma}
\newtheorem{prop}[thm]{Proposition}
\theoremstyle{definition}

\theoremstyle{remark}
\newtheorem{rem}[thm]{Remark}

\theoremstyle{example}

\numberwithin{equation}{section}

\begin{document}
\date{}
\title{A quadratic bound on the number of boundary slopes of essential surfaces with bounded genus}

\author{Tao Li 
\thanks{Partially supported by NSF grant DMS-0705285}\and 
Ruifeng Qiu 
\thanks{Partially supported by NSFC grant 10631060}\and 
Shicheng Wang
\thanks{Partially supported by NSFC grant 10625102}}

\maketitle

\begin{abstract}
Let $M$ be an orientable 3-manifold with $\partial M$ a single torus. We
show that the number of boundary slopes of immersed essential surfaces with
genus at most $g$ is bounded by a quadratic function of $g$. In the
hyperbolic case, this was proved earlier by Hass, Rubinstein and
Wang.

\textbf{Subject class:} 57M50, 57N10
\end{abstract}
\section{Introduction}

A proper immersion $f: (F, \partial F)\to (M, \partial M)$ from a compact surface to a  compact 3-manifold is
essential if it is $\pi_1$-injective and $\partial$-injective, i.e., it maps essential loops and arcs in $F$ to essential loops and arcs in $M$.
Let $M$ be a compact orientable 3-manifold with $\partial M$ a single torus.  We say a slope $s$ in $\partial M$ is realized by an essential surface if there is a proper essential immersion $f:
(F, \partial F)\to (M,
\partial M)$ such that every component of $f(\partial F)$ is a curve of slope $s$ in $\partial M$.  Such an immersed surface is particularly interesting because it extends to a closed immersed surface in the closed 3-manifold $M(s)$ obtained by Dehn filling along the slope $s$.

The study of boundary slopes of essential surfaces has been an active and
attractive topic for long times. Hatcher \cite{Ha} showed that there are only finitely many boundary slopes of embedded essential surfaces.   The number of boundary slopes of small-genus embedded surfaces (e.g.~punctured spheres or tori) is quite small and the study of these exceptional slopes is a center topic in the theory of Dehn surgery, see the survey article \cite{Go}.

However, for immersed essential surfaces, there is no such bound in
general. In fact, there are examples that every slope is realized by
an immersed essential surface, see  \cite{Ba, BC, O}. In \cite{HRW},
Hass, Rubinstein and Wang show that for hyperbolic manifolds, the
number of boundary slopes of essential surfaces of genus at most $g$
is bounded by $Cg^2$, where $C$ is a constant independent of the
manifold (see also Agol \cite{Ag}).  The purpose of this paper is to
extend the quadratic bound result to general 3-manifolds.

\begin{thm}\label{Tmain}
Suppose $M$ is an orientable 3-manifold with $\partial M$ a single
torus. For any $g$, let $N_g(M)$ be the number of slopes that can be realized by essential immersed surfaces of
genus at most $g$.

Then $N_g(M)\le\left\{
\begin{array}{cl}
C(M)g^2 & g\ge 1 \\
C'(M) & g=0
\end{array} \right.$ for some constants $C(M)$ and $C'(M)$ that depend on $M$.
\end{thm}

\noindent
{\bf Remark.}
(1). In \cite{HRW}, Hass, Rubinstein and Wang proved that $N_g(M)$ is finite, but no bound on $N_g(M)$ is given in \cite{HRW}.  Recently Zhang \cite{Zh} extended the techniques in \cite{HRW} and proved that $N_g(M)$ is bounded by $c(M)g^3$ for some constant $c(M)$ that depends on $M$.

(2).  The coefficient $C(M)$ depends on $M$.  One would hope for a quadratic bound independent of $M$, but even for embedded surfaces, it seems difficult to obtain such a bound if $M$ contains essential annuli.  Nevertheless, the coefficients $C(M)$ and $C'(M)$ can be algorithmically determined, see Remark~\ref{Rlast}.

(3). When $\partial M$ is a high genu surface, there are finiteness
and infiniteness results in both embedded and immersed case, see
\cite{SWu}, \cite{Qi}, \cite{HWZ}, \cite{La} and \cite{QW}.

\section{Some crucial facts}

The proof of Theorem \ref{Tmain} relies on a theorem of
Hass-Rubinstein-Wang \cite{HRW}, a theorem of Culler-Shalen
\cite{CS}, and  Li's extension of Hatcher's argument
\cite{Li2}. Propositions \ref{hyperbolic}, \ref{2-slopes} and
\ref{Lsign} below are their variations, presented in the
forms we need.

In this section we first consider a hyperbolic 3-manifold $M$ with possibly more than one cusp.  We denote by
$M_\text{max}$ the interior of $M$ with a system of maximal cusps
removed. Now we identify $M$ with $M_\text{max}$, then $\partial M$
has a Euclidean metric induced from the hyperbolic metric and each closed
Euclidean geodesic in $\partial M$ has length at least 1 (see
\cite{Ad} for detail).

\begin{prop}\label{hyperbolic}
Suppose $M$ is a hyperbolic 3-manifold as above and $T$ is a component of $\partial M$.  Suppose $F$ is an
essential immersed surface of genus $g$ in $M$ and let
$c_1,..., c_n$ be the components of $\partial F\cap T$.
\begin{enumerate}
\item If we identify $M$ with $M_\text{max}$, then
$$\Sigma _{i=1}^n L(c_i)\le -2\pi \chi(F),$$
where $L(c_i)$ is the length of an Euclidean geodesic homotopic to
$c_i$ in $T$.

\item Let $S$ be an embedded essential surface in $M$ and let $\gamma$ be a component of $\partial S\cap T$.  Then there is a number $C_S$ which can be expressed as an explicit function of $\chi(S)$, such that 
$$|\gamma\cap\partial F|\le -C_S\cdot\chi(F),$$ 
where $|\gamma\cap\partial F|$ is the minimum number of intersection points of $\gamma$ and $\partial F$.

\item There are two distinct  essential circles $\Gamma_1$ and
$\Gamma_2$ in $T$, such that
$$|\Gamma_j\cap\partial F| \le -C\chi(F)$$ for some constant $C$, where $|\Gamma_j\cap
\partial F|$ is the minimum number of intersection points of $\Gamma_j$ and $\partial F$
up to isotopy, $j=1,2$.

\end{enumerate}
\end{prop}

\begin{proof} Part (1) is proved in \cite{HRW}. 

Now we prove part (2).  Recall we have identified $M$ with $M_\text{max}$, and
$\partial M$ has a Euclidean metric induced from the hyperbolic
metric. We may assume that $\gamma$ and each component $c_i$ of $\partial F$ have been isotoped to be closed Euclidean geodesics in $\partial M$.

Let $p: E^2\to T$ be the universal cover, where $E^2$ is the
Euclidean plane. By lifting $\gamma$ to $E^2$, we get an Euclidean line segment $OO_1$ which projects to $\gamma$.  By part (1), the Euclidean length $L(\gamma)=L(OO_1)$ is at most $-2\pi\chi(S)$.  The covering translations of $O$ form a lattice in $E^2$. Let $O_2$ be a lattice point such that $OO_1$ and $OO_2$ span a fundamental parallelogram $P$ for $T$.  By a theorem of Cao and Meyerhoff (also see Lemma 2.2 of \cite{HRW}), $area(P)\ge 3.35$.

Let $h$ be the distance from $O_2$ to the line $OO_1$.  Since the Euclidean length $L(OO_1)\le -2\pi\chi(S)$ and since $area(P)\ge 3.35$, the height $h\ge\frac{3.35}{L(OO_1)}\ge \frac{3.35}{-2\pi\chi(S)}$.

By lifting $c_i$ to $E^2$, it is easy to see that the
length of $c_i$ is at least $h|c_i\cap \gamma|$.
By part (1), we have
$$h|\gamma\cap
\partial F|
=h \Sigma _{i=1}^n  |c_i\cap \gamma| \le \Sigma
_{i=1}^n L(c_i)\le -2\pi \chi(F).$$ 
So part (2) holds and $C_S=\frac{-4\pi^2\chi(S)}{3.35}$.

The proof of part (3) is similar.  Pick an origin $O$ in $E^2$ and consider the lattice $L$ in $E^2$ given by the covering translations of $O$. Let $O_1$ and $O_2$ be two independent
vertices in $L$ which have the first and second shortest distance
from the origin $O$. Let $\alpha$ be the angle of the triangle
$OO_1O_2$ at $O$ and let $l$, $l_1$, $l_2$ be the lengths of
$O_1O_2$, $OO_1$ and $OO_2$ respectively. By our assumptions above and by a theorem in \cite{Ad} mentioned earlier, we have $l\ge
l_2 \ge l_1\ge 1$.  This implies that $\alpha\ge \pi/3$. Furthermore, we can
assume that $\alpha\le \pi/2$, because otherwise we can replace one of the
vertices by its inverse.

$OO_1$ and $OO_2$ span a fundamental parallelogram $P$ for $T$.  It follows from our assumptions above that
$l_1sin\alpha$ (resp. $l_2sin\alpha$), the height of $P$ over $OO_2$
(resp. over $OO_1$), is at least $\frac{\sqrt 3} 2$.
Let $\Gamma_j=p(OO_j)$, $j=1,2$. As in part (2), we have
$$\frac{\sqrt 3} {2}|\Gamma_j\cap
\partial F|
=\frac{\sqrt 3} {2} \Sigma _{i=1}^n  |c_i\cap \Gamma_j| \le \Sigma
_{i=1}^n L(c_i)\le -2\pi \chi(F),$$ and part (3) follows with $C=\frac{4\pi}{\sqrt{3}}$.
\end{proof}

\begin{prop} \label{2-slopes}
Suppose $M$ is a hyperbolic 3-manifold as above and $T$ is a component
of $\partial M$.  Then $T$ has two distinct boundary slopes $c_1$ and
$c_2$ of embedded essential surfaces, i.e., there are
properly embedded essential surfaces $F_i$ in $M$ such that $F_i\cap
T$ is a multiple of $c_i$, $i=1,2$.
\end{prop}

\begin{proof} By performing hyperbolic Dehn filling on each boundary
component of $\partial M\setminus T$, we get a hyperbolic 3-manifold
$M^*$ with $\partial M^*=T$. By a theorem of Culler-Shalen \cite{CS},
there are two distinct boundary slopes $c_1$ and $c_2$ on $T$, i.e.
there are properly embedded essential surfaces $F_i^*$ in $M^*$ such
that $F_i^*\cap T$ is a multiple of $c_i$, $i=1,2$. So
$F_i=F_i^*\cap M$ has the required property.
\end{proof}

A surface in a Seifert fiber space is said to be horizontal if it is transverse to the $S^1$-fibers.
If an orientable Seifert fiber space has a single boundary component, then it is easy to see that all embedded horizontal surfaces have the same slope which is determined by its Euler number.  The following Lemma is a generalization of this fact to immersed horizontal surfaces in a Seifert fiber space with more than one boundary component.

\begin{prop}\label{Lsign}
Let $N$ be an orientable  Seifert fiber space with boundary and $T$
a boundary component of $N$.  Let $F_1$ and $F_2$ be immersed
essential horizontal surfaces in $N$.  Suppose $F_i\cap T$ is
embedded for both $i=1,2$ and $|\partial F_1\cap\partial F_2\cap T|$
is minimal in the isotopy classes of $F_1$ and $F_2$.  If there is a
double curve $\alpha\subset F_1\cap F_2$ with both endpoints in $T$,
then the curves of $F_1\cap T$ and $F_2\cap T$ must have the same
slope in $T$.
\end{prop}
\begin{proof}
The proof of the lemma is basically an argument first used by Hatcher in \cite{Ha} and then extended to immersed surfaces in \cite{Li2}.  As $N$ is a Seifert fiber space, we can fix a direction for the $S^1$--fibers of $N$ in $T$.
Since $N$ is orientable and each $F_i$ is horizontal, the normal direction of $\partial N$ and the orientation of the $S^1$--fibers in $T$ uniquely determine an orientation for every curve of $\partial F_1\cap T$ and $\partial F_2\cap T$.  Since $F_i\cap T$ is embedded, every  component of $\partial F_i\cap T$ ($i=1\ or\ 2$) with this induced orientation represents the same element in $H_1(T)$.  If $\partial F_1\cap T$ and $\partial F_2\cap T$ have different slopes, they must  have a nonzero intersection number.  Moreover, since we have assumed $|\partial F_1\cap\partial F_2\cap T|$ is minimal in the isotopy classes of $F_1$ and $F_2$, the signs of the intersection points of $\partial F_1\cap\partial F_2\cap T$ (with respect to the directions above) are the same, either all positive or all negative.

Let $\alpha\subset F_1\cap F_2$ be an intersection arc with both endpoints in $T$.  One can easily list all possible configurations of the directions of the $S^1$-fibers at $\partial\alpha$ and the induced orientations of $\partial F_1$ and $\partial F_2$.  However,  since each $F_i$ is horizontal, only two possible configurations can happen, see Figure~\ref{Fsign}.  In either case, the two ends of $\alpha$ give points of $\partial F_1\cap\partial F_2\cap T$ with opposite  signs of intersection.  This contradicts our conclusion on the sign of the intersection points above.  So $F_1\cap T$ and $F_2\cap T$ must have the same slope in $T$.

\begin{figure}[h]
\begin{center}
\psfrag{S1}{$F_1$} \psfrag{S2}{$F_2$} \psfrag{fd}{fiber direction}
\includegraphics[width=4.5in]{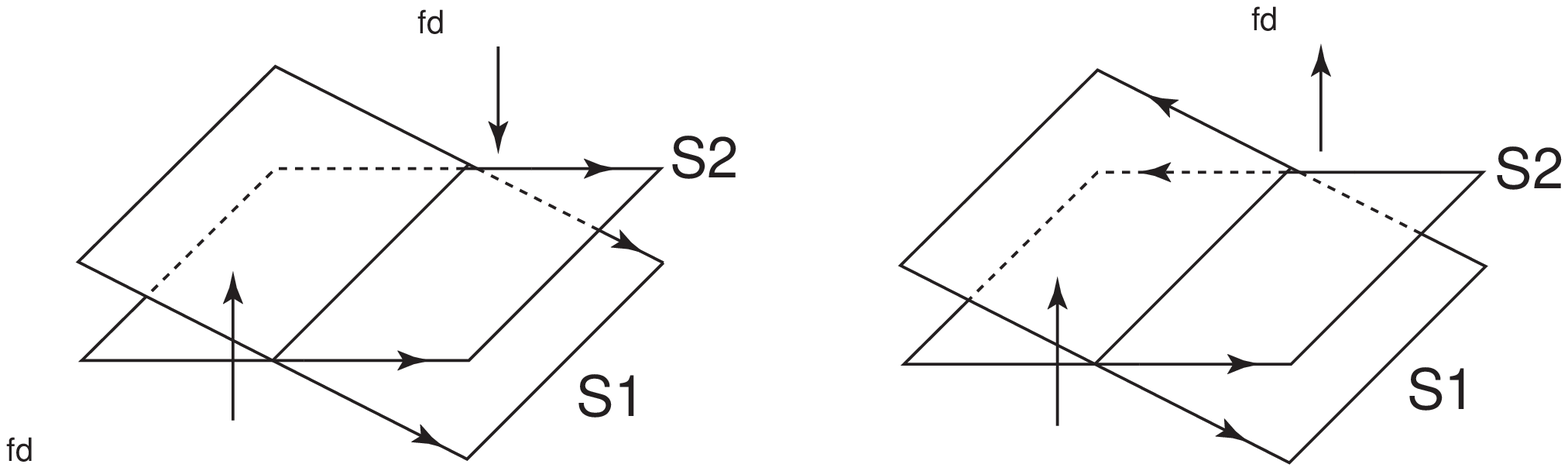}
\caption{} \label{Fsign}
\end{center}
\end{figure}
\end{proof}

The following fact follows immediately from Lemma~\ref{Lsign}.  Since an  essential surface in a Seifert fiber space is either vertical or horizontal \cite{H}, if $M$ is an orientable Seifert fiber space with a single boundary component, this means that only two possible slopes can be realized by immersed essential surfaces, one vertical and one horizontal.

\begin{col}\label{Csame}
Let $N$ be an orientable Seifert fiber space with a single boundary torus.  Then all immersed horizontal surfaces with respect to a fixed Seifert structure have the same slope in $\partial N$.
\end{col}

\section{Construct a surface of reference}\label{Scon}

Let $M$ be as in Theorem~\ref{Tmain}.
First note that we may assume $M$ is irreducible, since if $M$ is reducible we can use the prime factor of $M$ that contains $\partial M$ and the proof is the same.  Since the hyperbolic case is proved in \cite{HRW} and the Seifert fiber case is trivial (see Corollary~\ref{Csame}), we may assume $M$ has a nontrivial JSJ decomposition.

Let $\mathcal{T}$ be the set of JSJ
decomposition tori of $M$.  We call the closure (under path metric)
of each component of $M-N(\mathcal{T})$ a JSJ piece. Let $M_0$ be
the JSJ piece that contains the torus $\partial M$.

In this section, we suppose $M_0$ is a Seifert fiber space and  we
will use the JSJ structure of $M$ to construct a surface of
reference for counting the boundary slopes of immersed essential
surfaces.  This surface is in $M_0$ and is not a proper surface in
$M$.

 For any Seifert fiber space $N$
with boundary, we call a slope in a boundary torus the
\emph{vertical slope} if it is the slope of a regular fiber of $N$.

\begin{prop}\label{Phor}
Let $N$ be a Seifert fiber space and $T_0$, $T_1,\dots T_n$ the
boundary tori of $N$.  Let $s_i$ ($i=1,\dots n$) be any slope in $T_i$
that is not vertical in $N$.  Then there is an embedded horizontal surface
in $N$ realizing each slope $s_i$ in $T_i$.
\end{prop}
\begin{proof}
We perform Dehn fillings along each slope $s_i$ ($i=1,\dots n$) and let $\hat{N}$ be the resulting manifold.  So $\partial \hat{N}=T_0$. Since $s_i$ is not vertical in $N$, the Seifert structure of $N$ extends to $\hat{N}$.  Hence $\hat{N}$ is a Seifert fiber space with boundary.  Every Seifert fiber space with boundary has an embedded horizontal surface.  The restriction of a horizontal surface of $\hat{N}$ to $N$ is a horizontal surface of $N$ realizing each slope $s_i$ in $T_i$ ($i=1,\dots n$).
\end{proof}

Let $M_0$ be the Seifert JSJ piece of $M$ as above.  Let $\partial M$, $T_1,\dots, T_n$ be the boundary tori of $M_0$.  So each $T_i$ can be viewed as a JSJ torus in $\mathcal{T}$ and $M_0$ is a JSJ piece on one side of $T_i$.  Next we fix a slope in $T_i$ according the JSJ piece on the other side of $T_i$.  Let $M_i$ be the JSJ piece on the other side of $T_i$.  Note that $M_i$ is the same as $M_0$ if $T_i$ is glued to some $T_j$ in $M$.  We fix a slope $s_i$ for each boundary component $T_i$ of $M_0$ as follows.

\vspace{10pt}

\noindent
\emph{Case 1}. $M_i$ is a Seifert fiber space and $M_i$ is not a twisted $I$-bundle of a Klein bottle.  In this case we choose the slope $s_i$ of $T_i$ to be the slope of a regular fiber of $M_i$.  Note that $s_i$ is not a vertical slope for $M_0$, because otherwise the regular fibers of $M_0$ and $M_i$ match and $M_0\cup_{T_i} M_i$ is a Seifert fiber space, which contradicts the hypothesis that $T_i$ is a JSJ torus.  So $s_i$ is not a vertical slope for $M_0$.

\vspace{10pt}

\noindent
\emph{Case 2}.  $M_i$ is a twisted $I$-bundle over a Klein bottle.  In this case, $M_i$ has two different Seifert structures \cite{Ja}.  For any point $x\in T_i=\partial M_i$, we define $p(x)$ to be the other endpoint of the $I$-fiber of $M_i$ that contains $x$.  Let $\gamma_\nu$ be a simple closed curve in $T_i$ which is a regular fiber of $M_0$.  Let $s_i$ be the slope of $p(\gamma_\nu)$.  Note that $\gamma_\nu$ and $p(\gamma_\nu)$ bound an immersed essential annulus in $M_i$.  If $\gamma_\nu$ and $p(\gamma_\nu)$ have the same slope in $T_i$, i.e. $\gamma_\nu\cup p(\gamma_\nu)$ bounds an embedded annulus, then we can choose a Seifert structure for $M_i$ \cite{Ja} so that $\gamma_\nu$ is also a regular fiber for $M_i$ and hence $M_0\cup M_i$ is a Seifert fiber space, a contradiction to the hypothesis that $T_i$ is a JSJ torus.  So $s_i$ is not a vertical slope for $M_0$.

\vspace{10pt}

\noindent \emph{Case 3}. $M_i$ is hyperbolic.  By Proposition
\ref{2-slopes}, $T_i$ has at least two boundary slopes (of embedded
essential surfaces in $M_i$). In this case we choose $s_i$ to be a
boundary slope of $M_i$ that is not a vertical slope in $M_0$.  So
there is an embedded essential surface $S_i$ in $M_i$ whose boundary
in $T_i$ has slope $s_i$ and $s_i$ is not a vertical slope in $M_0$.

\vspace{10pt}

By Proposition~\ref{Phor}, $M_0$ contains a properly embedded horizontal surface $S$ such that the slope of $\partial S\cap T_i$ is the slope $s_i$ described above.  Note that $S$ is not a properly embedded surface in $M$, since two tori $T_i$ and $T_j$ ($i\ne j$) may be glued together in $M$ and $s_i$ and $s_j$ may not match in the corresponding JSJ torus of $M$.

Next we fix the surface $S$ in the construction above.  Let $\mu$ the slope of $S\cap\partial M$ in the torus $\partial M$ and let $\nu$ be the vertical slope of $\partial M$ with respect to the Seifert structure of $M_0$.

\section{Proof of the Theorem \ref{Tmain}}

Let $F$ be a proper immersed essential surface of genus $g$ in $M$.

If $M_0$ is hyperbolic then  Theorem~\ref{Tmain} follows from
\cite{HRW}. More precisely, suppose $\partial F$ is an $n$ multiple
of a slope $c$ in $\partial M$ and  we have identified $M_0$ with
the metric space $M_{0 \text{max}}$ as in Proposition \ref{hyperbolic}. By Proposition \ref{hyperbolic}
(1), we have
$$n L(c)\le L(\partial(F\cap M_0))\le -2\pi \chi(F\cap M_0)\le -2\pi \chi(F) = 2\pi(2g-2+n).$$
Then as discussed in \cite{HRW} we have $L(c)\le 2\pi$ if $g=0$ and
$L(c)\le 2g\pi$ if $g>0$, therefore $N_g(M)\le C'$ for $g=0$ and
$N_g(M)\le Cg^2$ for some constants $C'$ and $C$ independent of $M$.

Below we assume that $M_0$ is a Seifert fiber space.

We may assume the slope of $\partial F$ is not the vertical slope of
$M_0$, so $F\cap M_0$ is horizontal in $M_0$.  Since $\partial M$ is
incompressible, $F$ is not a disk.  If $F$ is an annulus, then
$F\cap M_0$ is a horizontal annulus.  The only orientable Seifert
fiber space that admits a horizontal annulus is either
$T^2\times I$ or a twisted $I$-bundle over a Klein bottle.  Since
$M_0$ is a JSJ piece, $M_0$ is not $T^2\times I$.  If $M_0$ is a
twisted $I$-bundle over a Klein bottle, $M_0=M$ and by
Corollary~\ref{Csame} there are only two possible slopes for $F$.  Thus
Theorem~\ref{Tmain} holds if $\chi(F)\ge 0$. So in this section, we
assume $\chi(F)<0$.

\begin{lem}\label{Lver}
Let $N$  be a Seifert JSJ piece of $M$ and $v$ a regular fiber of $N$. Suppose $N$ is not a twisted $I$-bundle over a Klein bottle.  Let $F$ be an essential surface in $M$ and suppose $F\cap N$ is horizontal in $N$.  Then $|v\cap F|\le -6\chi(F)$.
\end{lem}
\begin{proof}
Let $O(N)$ be the base orbifold of $N$.  Since $O(N)$ has  boundary  and $N$ is not a solid torus, $\chi(O(N))\le 0$. Moreover, since $N$ is orientable and is not $T^2\times I$, $\chi(O(N))=0$ if and only if $O(N)$ is a disk with two cone points both of order 2 and $N$ is a twisted $I$-bundle over a Klein bottle.
Thus by our hypothesis that $N$ is not a twisted $I$-bundle over a Klein bottle, we have $\chi(O(N))<0$.

Since $F\cap N$ is horizontal in $N$, $\chi(F\cap N)=k\chi(O(N))$ where $k=|v\cap F|$.  Since $O(N)$ has boundary, the maximal possible value for $\chi(O(N))$ occurs when $O(N)$ is a disk with two cone points of orders 2 and 3 respectively, in which case $\chi(O(N))=-1/6$.  Therefore $\chi(O(N))\le -1/6$ and $k=|v\cap F|\le -6\chi(F\cap N)\le -6\chi(F)$.
\end{proof}
\begin{rem}\label{R2}
In the proof of Lemma~\ref{Lver}, $\chi(O(N))\le -1/2$ except when $O(N)$ is a disk with two cone points.  Thus $|v\cap F|\le -2\chi(F)$ if $\partial N$ has more than one boundary component.  This is a key observation in the proof of the following lemma, see \cite{Zh}.
\end{rem}

\begin{lem}[\cite{Zh}, Lemma 3.2]\label{Lzhang}
Let $M$ and $M_0$ be as in section~\ref{Scon} and let $\nu$ be the vertical slope of $\partial M$ in $M_0$.  Let $F$ be an immersed essential surface in $M$ of genus at most $g$ and let $s_F$ be the boundary slope of $F$ in $\partial M$.  Then the geometric intersection number $$\Delta(\nu,s_F)\le U(g)=\left\{
\begin{array}{cl}
2 & g=0 \\
2g & g\ge 1
\end{array} \right. .$$
\end{lem}
\qed

\begin{proof}[Proof of Theorem \ref{Tmain} when $M_0$ is a Seifert fiber space]
Let $S$ be the fixed embedded horizontal surface in $M_0$
constructed in section~\ref{Scon}. Let $F$ be an immersed essential
surface in $M$ of genus at most $g$. We will study the intersection
of $F\cap M_0$ and $S$.  Let $s_F$ be the boundary slope of $F$.
Our main goal is to show that $\Delta(\mu, s_F)$ is bounded by a
linear function of $g$, where $\mu$ is the slope of $\partial
S\cap\partial M$.  As only one slope is vertical, we suppose $F\cap M_0$ is horizontal in $M_0$.

We will use the same notation as section~\ref{Scon}. The boundary tori of $M_0$ are $\partial M$, $T_1,\dots, T_n$ and $S$ is properly embedded in $M_0$.  In this section, we view $S$ as a surface in $M$ instead of $M_0$.  Since it is possible that $T_i$ and $T_j$ ($i\ne j$) are glued together in $M$, when regarded as a surface in $M$, curves of $\partial S$ may intersect in a JSJ torus of $M$.

Now we consider the intersection of $F$ and $S$.  A key difference between $F$ and $S$ is that $F$ is a proper surface in $M$ while $S$ is only defined in $M_0$.  We view the torus $T_i$ as a JSJ torus of $M$ and as in section~\ref{Scon}, let $M_i$ be the JSJ piece incident to $T_i$ on the other side of $M_0$ ($M_i$ may be the same JSJ piece as $M_0$).  Let $\Gamma_i=S\cap T_i$ in $M_0$.  As above, we view $\Gamma_i$ as a collection of curves in a JSJ torus in $M$.  Next we estimate $|F\cap\Gamma_i|$.  Let $k_i$ be the number of components of $\Gamma_i$. As in the construction of $S$, we have 3 cases:

\vspace{10pt}

\noindent
\emph{Case 1}.  $M_i$ is a Seifert fiber space and $M_i$ is not a twisted $I$-bundle over a Klein bottle.  By the construction of $S$, in this case, each curve in $\Gamma_i$ is a regular fiber of the Seifert fiber space $M_i$.  By Lemma~\ref{Lver}, $|F\cap\Gamma_i|\le -6k_i\chi(F\cap M_i)\le -6k_i\chi(F)$.

\vspace{10pt}

\noindent
\emph{Case 2}.  $M_i$ is a twisted $I$-bundle over a Klein bottle.  By our construction of $S$ in this case, each curve $\gamma$ in $\Gamma_i$ and a regular fiber $p(\gamma)$ of $M_0$ bound an immersed essential annulus in $M_i$.  We may assume $F\cap M_i$ to be essential in $M_i$. So the intersection of $F$ and an essential annulus in $M_i$ consists of essential arcs in the annulus.  In particular, $|F\cap\gamma|=|F\cap p(\gamma)|$.   Since $p(\gamma)$ is a regular fiber of $M_0$ for each curve $\gamma$ in $\Gamma_i$ and since $M_0$ has more than one boundary component, by Lemma~\ref{Lver} and Remark~\ref{R2}, $|F\cap\Gamma_i|\le -2k_i\chi(F\cap M_0)\le -2k_i\chi(F)$.

\vspace{10pt}

\noindent \emph{Case 3}. $M_i$ is hyperbolic. In this case there is
an embedded essential surface $S_i$ in $M_i$ whose boundary slope in the
torus $T_i$ is the same as the slope of $\Gamma_i$.  Now we consider
the intersection of $S_i$ and $F\cap M_i$. By Proposition
\ref{hyperbolic} (2), $|F\cap\Gamma_i|\le -c_i\chi(F\cap M_i)\le
-c_i\chi(F)$ for some number $c_i$ which depends on $|\Gamma_i|$ and $\chi(S_i)$.

\vspace{10pt}

Let $\Gamma_0=\partial S\cap\partial M$ be the boundary curves of $S$ lying in $\partial M$.  So $\partial S-\Gamma_0=\bigcup_{i=1}^n\Gamma_i$.  By the argument above, there is a number $c>0$ depending on $S$ such that the total number of intersection points of $F$ and $\partial S-\Gamma_0$ is at most $-c\chi(F)=c(2g-2+|\partial F|)$ for some constant $c$ which depends on $|\partial S-\Gamma_0|$ and the surface $S_i$ in the case that $M_i$ is hyperbolic as in Case (3).

Let $\Delta=\Delta(\mu, s_F)$ be the intersection number of a curve in $\Gamma_0$ and a curve $\partial F$.  So the total number of intersection points of $\Gamma_0$ and $\partial F$ is $\Delta\cdot|\Gamma_0|\cdot|\partial F|$.

Thus, if $g\ge 1$, there is a number $C_1$ depending on $S$ and $S_i$ such
that if $\Delta> C_1g$, we have $\Delta\cdot|\Gamma_0|\cdot|\partial
F|>c(2g-2+|\partial F|)$ and hence there must be an arc in $S\cap F$
with both endpoints in $\partial M$.  However, by
Proposition~\ref{Lsign}, this means that $\partial F$ has the same
slope as $\partial S\cap\partial M$ and $\Delta=0$, a contradiction.
Therefore, if $g\ge 1$, $\Delta\le C_1g$ for some constant $C_1$ which depends on $S$ and the surface $S_i$ in Case (3).

Similarly if $g=0$, there is a number $C_0$ such that if $\Delta>C_0$, then $\Delta\cdot|\Gamma_0|\cdot|\partial F|>c(|\partial F|-2)$ and hence there must be an arc in $S\cap F$ with both endpoints in $\partial M$, which means that $\Delta=0$.  Thus if $g=0$, $\Delta\le C_0$ for some constant $C_0$ which depends on $M$.

We have two fixed slopes for $\partial M$, the vertical slope $\nu$ and the slope $\mu$ of $\partial S\cap\partial M$.  For any horizontal immersed essential surface $F$ of genus at most $g$, let $s_F$ be its boundary slope. The argument above says that $\Delta(\mu,s_F)\le V(g)$, where $V(g)=C_1g$ if $g\ge 1$ and $V(g)=C_0$ is $g=0$ for some constants $C_1$ and $C_0$ depending on $S$.  By Lemma~\ref{Lzhang}, $\Delta(\nu, s_F)\le U(g)$ where $U(g)=2g$ if $g\ge 1$ and $U(g)=2$ if $g=0$.  Therefore,  the total number of possible slopes for $\partial F$ is bounded by a quadratic function of $g$, where the coefficients depend on the fixed surface $S$ and the surface $S_i$ used in the hyperbolic JSJ piece as in Case (3).
\end{proof}
\begin{rem}\label{Rlast}
If one uses part (3) of Proposition~\ref{hyperbolic} instead of part (2) in the argument, then one can prove the main theorem without using the Culler-Shalen theorem (i.e. Proposition~\ref{2-slopes}).  However, there is an advantage of using Proposition~\ref{2-slopes}.  Given any triangulation of a 3-manifold, one can use normal surface theory to algorithmically find two embedded essential surfaces with different boundary slopes whose existence is guaranteed by Proposition~\ref{2-slopes}.  Since there are algorithms to determine the JSJ and Seifert structures, the constant in Theorem~\ref{Tmain} can be found algorithmically by following the proof.
\end{rem}

\bibliographystyle{amsalpha}

\noindent
Department of Mathematics; Boston College; Chestnut Hill, MA 02167 USA. \\
Email address: taoli@bc.edu

\noindent
Department of Mathematics; East China Normal University; Shanghai
200062 CHINA \\
Email address: qiurf@dlut.edu.cn

\noindent 
Department of Mathematics; Peking University, Beijing 100871, CHINA \\
Email address: wangsc@math.pku.edu.cn

\end{document}